\documentclass[11pt,reqno]{amsart}
\usepackage{amsmath,amsfonts,amssymb,amsthm,amsbsy,bbm, amscd,  epsf,calc,enumerate,color,graphicx,verbatim}
\usepackage{ifpdf}
\ifpdf
\pdfoutput=1
    \usepackage[colorlinks,citecolor=blue,linkcolor=blue]{hyperref}
\fi
\usepackage{color}
\usepackage{datetime}
\numberwithin{equation}{section}
\newcommand{\commentout}[1]{}

\newcommand{\R}{\mathbb{R}}

\newcommand{\T}{\mathbb{T}}
\newcommand{\beq}{\begin{equation}}
\newcommand{\eeq}{\end{equation}}
\newcommand{\ol}{\overline}
\newcommand{\wh}{\widehat}

\newcommand {\e}  {\varepsilon}

\newcommand {\ep}  {\epsilon}

\newcommand{\bea} {\begin{array}{rl}}
\newcommand{\eea} {\end{array}}
\newcommand{\bepa}{\left\{ \begin{array}{l}}
\newcommand{\eepa} {\end{array}\right.}
\newtheorem{theorem}{Theorem}[section]

\title{ {Homogenization of the backward-forward mean-field games systems in   periodic environments}}

 \author{
Pierre-Louis Lions$^{1}$ and Panagiotis E. Souganidis$^{2,3}$}

\begin{document}

\maketitle
\pagestyle{plain}
\pagenumbering{arabic}

\begin{center}{ \today}\end{center}

\begin{abstract}
We study the homogenization properties in the small viscosity limit and  in periodic environments of the (viscous) backward-forward mean-field games system. We consider separated Hamiltonians and provide results for systems  with (i) ``smoothing'' coupling and general initial and terminal data,  and (ii) with ``local coupling'' but well-prepared data.The limit  is a first-order forward-backward system. In the nonlocal coupling case,  the averaged system is of mfg-type, which is well-posed in some cases. For the problems with local coupling, the homogenization result is proved assuming that the formally obtained limit system has smooth solutions with well prepared initial and terminal data. It is also shown, using a very general example (potential mfg), that the limit system is not necessarily of mfg-type. 
\end{abstract}
\bigskip

\\
\\
 {\bf AMS Class. Numbers} 35B27, 35B40, 35K40, 35K59, 91A13
\bigskip
\section{Introduction}

This paper is the first step of a general program to study the homogenizing properties of mean-field games (for short mfg) set in self-averaging, for example,  periodic or stationary ergodic environments. 
\smallskip

In particular, we study  the homogenization properties,  in periodic environments, of the classical viscous backward-forward mfg-system

\begin{equation}\label{mfg1}
\begin{cases}
\partial_t u^\ep  - \ep \Delta u^\ep + H\left(Du^\ep, \dfrac{x}{\ep}\right) - F_\ep[m^\ep]=0 \ \text{in} \ \R^d\times (0,T) \quad  u^\ep(\cdot,0)=u_0, \\[2mm]
\partial_t m^\ep + \ep \Delta m^\ep + \text{div} \left[D_pH\left(Du^\ep, \dfrac{x}{\ep}\right) m^\ep\right]=0  \ \text{in} \ \R^d\times (0,\infty) \quad m^\ep(\cdot,T)=m_T.
\end{cases}
\end{equation}
which is the core of the mean field games theory without common noise.
\smallskip

In \eqref{mfg1},  $H=H(p,y):\R^d\times\R^d \to \R$ is periodic in  the second argument on the torus $\T^d$ and $D_pH$ is the derivative of $H$ with respect to its first argument. 
\smallskip

\smallskip



We discuss  next the forcing term $F_\ep$ in  \eqref{mfg1}. To present a unified setting, that is, to discuss at the same time both the cases of local and nonlocal (smoothing) dependence on the density, we assume that 
\beq\label{chi1}
F_\ep[m](x)=F(x, \dfrac{x}{\ep},m)
\eeq
where 
\beq\label{chi2}
F=F(x,y,m):\R^d\times\R^d\times C^2(\R^d) \to \R \ \text{is $\T^d$-periodic in $y$.} 
\eeq

When dealing with the local case, \eqref{chi2} follows from assuming that 
\beq\label{chi3}
F=F(x,y,m):R^d\times\R^d\times \R \to \R \ \text{is $\T^d$-periodic in $y$,}  
\eeq
in which case
\beq\label{chi4}
F_\ep[m](x)=F(x, \dfrac{x}{\ep},m(x)).
\eeq

When we consider the  nonlocal dependence on the density, for simplicity, we will omit the $y$-dependence of $F$ in 
\eqref{chi2}  and extend the domain of $F$ in \eqref{chi2}, that is, we assume that 
\beq\label{chi5}
F=F(x, m):\R^d\times \mathcal{P} \to \R,
\eeq
where $\mathcal{P}$ is the space of probability measures on $\R^d$. 
\smallskip

Most of the  results extend to mfg-systems with not separated Hamiltonians, that is $H$ that also depend on $m$. To keep, however,  the  notation and to explain the main ideas  in this note  we chose to work with \eqref{mfg1}.

\smallskip

Our main results are:
(i) Homogenization to a unique limit 
\eqref{mfg1} with nonlocal coupling and general data to an mfg-system.  
(ii) Homogenization for \eqref{mfg1} with local  coupling to a forward-backward system, when the latter has a smooth solution and \eqref{mfg1} has well-prepared data.  The limit system is not, however,  in general  of mfg-type.
\smallskip

In either case, the classical ansatz 
\beq\label{ansatz}
u^\ep(x,t)= \ol u(x,t) +\ep v(\frac{x}{\ep}) \ \  \text{and} \  \ m^\ep(x,t)=\ol m(x,t) \left( \mu(\frac{x}{\ep}) +\ep \nu (\frac{x}{\ep})\right) 
\eeq
with 
\beq\label{ansatz1}
 v, \mu, \nu: \R^d \to \R \ \text{are} \  \T^d-\text{periodic,} \ \mu>0 \ \text{and} \  \int_{\R^d} \mu(y) dy=1,
\eeq 
formally leads, for each $p\in \R^d$ and $m\in L^1(\R^d)$ or $m \in \R$,  to  the mfg-ergodic system (cell-problem) 

\begin{equation}\label{cell.1}
\begin{cases}
-\Delta v + H(Dv+p, y) - F[ m\mu]= \ol H(p,m,x) \ \text{in} \ \T^d, \\[1.5mm]
\Delta \mu + \text{div} \left[ D_p H(Dv+p,y)\mu\right]=0  \ \text{in} \ \T^d,\\[1.5mm]
\int_{\T^d} v(y) dy=0,  \  \mu > 0 \   \text{in} \ \T^d \ \text{and} \  \int_{\T^d} \mu(y) dy=1,
\end{cases}
\end{equation}
\smallskip

where $y=x/\ep$ and $F[m\mu]$ is defined as in \eqref{chi1}, \eqref{chi2} with $\ep=1$. Note that the dependence of $\ol H$ on $x$ is through  $F$.
\smallskip

It is known (see Lasry and Lions \cite{LL3}, Lions \cite{LCdF}, and Cardaliaguet, Lasry, Lions and Porretta \cite{CLLP1, CLLP2}) that, under the appropriate conditions on $H$ and $F$ to be introduced later, for each $x, p \in \R^d$ and $m: \R^d \to \R$ smooth, \eqref{cell}, there exists a unique  constant $\overline H(p,x,m)$ such that \eqref{cell.1}  has a unique $\T^d$-periodic solution $(v,m)=(v(\cdot; p, x, m), \mu(\cdot; p, x, m)) \in C^2(\R^d) \times L^1(\R^d).$
\smallskip

The homogenized system is then 
\begin{equation}\label{mfg2}
\begin{cases}
\partial_t \overline u + \overline H\left(D\overline u, x,  \overline  m \right)=0  \ \text{in} \ \R^d\times (0,T] \qquad  \overline u(\cdot,0)=u_0,\\[1.5mm]
 \partial_t \overline m + \text{div}\left[ \overline b(D\ol u, x, \ol m) \overline m\right ]=0 \ \text{in} \ \R^d\times [0, T) \qquad \overline m(\cdot,T)=m_T, 
\end{cases}
\end{equation}

where, for $p\in \R^d$ and $m\in L^1(\R^d)$  or $m\in \R$,
\begin{equation}\label{mfg2.1}
\ol b(p,x, m)=\int_{\T^d} D_pH(p + Dv(y; p,x,m),y) \mu (y; p, x, m) dy. 
\eeq
\vskip.075in

An important question  is whether \eqref{mfg2} is actually an mfg-system, that is, whether, for all $p\in \R^d$ and $m\in L^1(\R^d)$ or $m \in \R$,
\beq\label{mfg10}
\ol b(p, x, m)= D_p \overline H(p, x,  m).
\eeq

We state now in an informal way, that is, without  precise assumptions, the two main convergence results. The first is about the nonlocal setting and the second about the local one. 

\begin{theorem}\label{takis100} 
Assume that  \eqref{mfg1} is nonlocal, that is, $F_\ep$ is given by \eqref{chi4} and \eqref{chi5}.
Then, along subsequences $\ep \to 0$, locally uniformly in $\R^d\times [0,T]$ and in $L^1(\R^d\times (0,T))$, the solution $(u^\ep,m^\ep)$ of \eqref{mfg1} converges to a solution $(\ol u, \ol m)$ of \eqref{mfg2}, which is always of mfg-type. If, in addition,  \eqref{mfg1} satisfies the conditions that make it well-posed, then \eqref{mfg2} is also well-posed and the full family converges to the unique solution of \eqref{mfg2}. 
\end{theorem}

\smallskip
\begin{theorem}\label{takis1000}
Assume that  the coupling in \eqref{mfg1} is local, that is, $F_\ep$ satisfies  \eqref{chi1} and \eqref{chi2}. If \eqref{mfg2} has a smooth solution $(\ol u, \ol m)$ and \eqref{mfg1} has  well prepared initial and terminal data, then, as $\ep \to 0$, locally uniformly in $\R^d\times [0,T]$ and in $L^1(\R^d\times (0,T))$, the solution $(u^\ep,m^\ep)$ of \eqref{mfg1} converges to $(\ol u, \ol m)$. In general, \eqref{mfg2} is not an mfg-type system. 
\end{theorem}
%
%
\smallskip

A particular case of Theorem~\ref{takis1000} is when, for some $p\in \R^d$, 
\beq\label{well}
u_0(x)= (p,x)  \ \text{and} \ m_T\equiv1.
\eeq
In this case, $(\ol u, \ol m)$ with  $\ol u(x,t)= (p,x) - t\ol H(p,1)$ and $m\equiv 1$ solves $\eqref{mfg2}$, and an argument similar to the proof of Theorem~\ref{takis100} gives that, as $\ep \to 0$ and locally uniformly in $\R^d\times [0,T]$ and  in $L^1(\R^d\times (0,T))$, 
\beq\label{well1}
u^\ep(x,t) \to (p,x) - \ol H(p,1) \ \text{ and} \  m^\ep \to 1.
\eeq

Finally, we note that after the results of this paper were obtained and announced, the authors became aware of a  work by Cesaroni, Dirr and Marchi \cite{CDM} which is about  a special class of  \eqref{mfg1} with quadratic $H$ and set in $\T^d\times [0,T]$.

\subsection*{Background} The theory of mean-field games was introduced  about a decade ago with by Lasry and Lions \cite{LL3}, who developed the fundamental  elements of the mathematical theory, and, independently, by  Huang, Malham\'e, and Caines  \cite{HMC2}  who considered  a particular class of mfg.  Since then,  the subject has grown  rapidly  both in terms of theory and applications. 
Some of the landmark theoretical results include the introduction of the so-called ``Master equation,''  the theory of ``monotone hyperbolic systems'' for finite state space models, and the ``Hilbertian'' approach for infinite dimensional problems.  References for these developments are the courses of 
Lions at  College de France \cite{LCdF}, which are available on line, and  the forthcoming books of Carmona and  Delarue \cite{CD1, CD2}.  
A partial and by no means complete  list of references of the earlier work in this general area  includes (in alphabetical order)   Achdou, Buera, Lasry, Lions and Moll \cite{ABLLM}, Achdou, Giraud, Lasry and Lions \cite{AGLL},  Cardalliaguet, Delarue, Lasry and Lions \cite{CDLL}, Cardaliaguet, Lasry, Lions and Porretta \cite{CLLP1, CLLP2}, Carmona and Delarue \cite{CD3}, Gabaix, Lasry, Lions and Moll \cite{GLLM}, Gu\'eant, Lasry and Lions \cite{GLL}, Huang, Caines and Malham\'e \cite{HMC1}, Huang, Malham\'e, and Caines  \cite{HCM1, HCM2, HCM3},  Lachapelle, Lehalle, Lasry and Lions \cite{LLLL}, 
and Lasry and Lions \cite{LL1, LL2, LL4}. 

\vskip.1in

\noindent  Applications that have been so far looked at  range from  complex socio-economical topics, regulatory financial issues,  crowd movement,  meaningful ``big'' data and advertising   to engineering contexts involving ``decentralized intelligence'' and machine learning;  two concrete examples being  fleets of automated cars  and future telecommunication networks. 
\vskip.1in
 
 \noindent  At the beginning,  MFG models were introduced to describe the behavior of a large group or several groups of agents using a mean field approach as in statistical physics. Such models can be derived rigorously from $N$-players systems as $N$ tends to infinity.  In this context, an agent is someone trying to optimize  certain criteria which, together with its dynamics, depend on the other agents and their actions. It is important to emphasize that the agents  react,  anticipate and strategize instead of simply reacting instantaneously. The latter is, for example,  the case in many agent-based models or ones derived from statistical mechanics and/or kinetic considerations.  
\vskip.1in
 
 \noindent  MFG are the ideal mathematical structures to study the quintessential problems in the social-economical sciences, which differ from physical settings because of the  forward looking behavior on the part of individual agents. Concrete examples of applications in this direction include  the modeling of the macroeconomy and  conflicts in the modern era. In both cases, a large number of agents interact strategically in a stochastically evolving environment, all responding to partly common and partly idiosyncratic incentives, and all trying to simultaneously forecast the decision of others.

\subsection*{Future work} This note is the first in a series of works that will provide a systematic homogenization theory for mean-field games. Among others,  we will consider general initial and terminal data for \eqref{mfg1}, we will study the well-posedness of the limit system when it is not of mfg-type, and, finally, we will consider extensions to 
random media as well as to mfg with variational structure and even  common noise. 

\subsection*{Organization of the paper}  In section~2 we discuss the ansatz,  we derive formally both the cell problem and the homogenized systems, and state the result about the solvability of the cell problem.  In section~3 we concentrate on the nonlocal problem.  In section~4, we study the local problem. We show that it homogenizes to a system which is not always of mfg-type. 
%

\subsection*{Notation} Throughout the paper $C_c(\R^d)$ denotes the space of compactly supported continuous  functions on $\R^d$. When more regularity is needed, we simply say smooth functions in $C_c(\R^d)$.
We also write 
$C_{c,p}(\R^d \times \T^d)$ for the space of continuous functions which are compactly supported in the first argument and $\T^d$-periodic in the second, and we use the same convention when more regularity is needed. Finally, $C_b(U;V)$ is the space of continuous bounded functions defined on $U$ and values in $V$. 

\subsection*{Acknowledgment} The first author was  partially supported by the Air Force Office for Scientific Research  grant FA9550-18-1-0494 and the Office for Naval Research grant N000141712095.
smallskip

The second author was   partially supported by the National Science Foundation grant DMS-1600129, the Office for Naval Research grant N000141712095 and the Air Force Office for Scientific Research grant FA9550-18-1-0494.

\section{The assumptions}

As far as $H:\R^d\times\R^d \to \R$ is concerned, throughout the paper we assume that 

\beq\label{periodic}
\text{$y \to H(p,y)$ is periodic in $\T^d$  for each $p \in \R^d$,}
\eeq
\beq\label{convex}
\text{$p\to H(p,y)$ is uniformly convex  in $p$ for each $y\in \T^d$,} 
\eeq 
%
and 
\beq\label{bound}
\text{$H$ and $D_pH$ are uniformly bounded in $y$ for $p$ bounded in $\R^d$;}
\eeq
%
%
%

\smallskip

The assumptions on $F$ and the joint dependence of $H$ and $F$ in $x$ and $y$ depend on the type of the coupling we consider.

\subsection*{The nonlocal coupling}
Throughout the discussion of the mfg-systems with nonlocal coupling,  we assume that  the coupling $F_\ep$ is independent of $\ep$, that is 
\beq\label{2.1}
F_\ep[m](x)=F[m](x)=F(x,m)
\eeq
where $F: \R^d\times\mathcal{P} \to \R$ is
%
%
\beq\label{FLip}
\text{$x\to F(x,m)$ is Lipschitz continuous uniformly in $x\in \R^d$ and $m\in \mathcal{P}$,}
\eeq
and
\beq\label{cm0}
\text{ $m\to F(x, m)$ is  
continuous  with respect to the  weak topology of measures uniformly in $x$}
\eeq
that is, 
\beq\label{cm}
\text{if}  \ 
m^\ep \underset{\ep\to 0 }\rightharpoonup m  \  \text{in \ the sense of measures, then, uniformly in $x$, $F(x, m^\ep) \underset{\ep\to 0 } \to F(x, m)$.}
\eeq

In order to obtain gradient bounds for $u^\ep$ in \eqref{mfg1}, it is necessary to assume some additional conditions on $H$ and $F$.  As far as the latter is concerned, we assume that 
\beq\label{2.2}
m \to F[m] \ \text{is differentiable in $m$ and $F'[m]$ is smooth and bounded in $\mathcal{P}$}
\eeq
that is,
\beq\label{2.3}
\|D_xF'[m]\|_\infty + \|D^2_xF'[m]\|_\infty \ \text{is bounded in $\mathcal{P}$.}
\eeq
We also need to assume that 
%
%
%
%
%
%
\beq\label{LIP}
\begin{cases}
\text{there exists $\theta \in (0,1)$ such that, for $\ep \in (0,1)$, large $|p|$ and $m \in \mathcal{P}$,}\\[2mm]
\underset{x\in \R^d, \ y \in \T^d}\inf \left \{ \theta H^2(p,y) +  d \left(D_y H(p,y), p \right ) -\ep d   \left(D_x F(x,m), p\right) \right \} >0. 
\end{cases}
\eeq

As far as  the well-posedness of \eqref{mfg1} goes, it was shown in \cite{LL3, LCdF} that \eqref{mfg1} has a unique solution,  if,  either

\beq\label{mon}
\text{ $m \to F[m]$ is monotone in $\mathcal{P}$,}
\eeq
that is, 
\beq\label{mon1}
\int_{\R^d} \left(F(x, m_1) -F(x, m_2)\right) (m_1(x)-m_2(x))dx\geq 0 \ \text{for all} \ m_1, m_2\in \mathcal{P}, 
\eeq
and 
\beq\label{sconvex}
H \ \text{is strictly convex in $p$ uniformly in $y\in \T^d$ },
\eeq
that is, 
\[\text{if} \ H(p+q,y) - H(p,y) - \left( D_pH(p,y), q\right) =0, \ \text{then} \ q=p,\]
or 
\beq\label{smon}
\text{ $F$ is strictly monotone in $\mathcal{P}$,}
\eeq
that is, 
\beq\label{smon1}
\text{if}  \ \int_{\R^d} \left(F(x, m_1)-F(x,m_2)\right) (m_1-m_2)(x)dx= 0, \ \text{then} \ m_1=m_2.
\eeq

\smallskip

%
%

\subsection*{The local coupling} When the coupling is local, that is
\[F_\ep[m](x)=F(x,\dfrac{x}{\ep},m(x)),\]
with $F$ as in \eqref{chi3} and 
\beq\label{local}
F \in C^2(\R^d\times \R^d\times \R).
\eeq
Since for the result we assume that the formally derived limiting system has smooth we do not need  uniform in $\ep$ Lipschitz bounds on $u^\ep$, and, hence, assumptions like  \eqref{2.2}, \eqref{2.2} and \eqref{LIP}.
\smallskip

For the well-posedness of \eqref{mfg1}, it is necessary, as in the nonlocal case, to assume that 
\beq\label{2.8}
H \ \text{is strictly convex and, for all $\ep>0$} 
\eeq
and 
$F_\ep$ is strictly monotone, that is, there exists $c>0$ such that, for each $\ep>0$ and any smooth functions $m_1, m_2$,
\beq\label{takis100}
\int_{\R^d} (F_\ep[m_1](x)-F_\ep[m_2](x))(m_1(x)-m_2(x)) dx\geq c \int_{\R^d} |m_1(x)-m_2(x)| dx.
\eeq

%
%
In view of the form of $F_\ep$, we explain next a condition that $F$ in \eqref{chi3} must satisfy in order for $F_\ep$ to be monotone in the sense of \eqref{takis100}. For simplicity we assume that $c=0$ in \eqref{takis100}. 

\smallskip
Recall that $F_\ep$ is monotone for each fixed $\ep>0$, if, for  all smooth $f_1, f_2:\R^d\to \R$, 
\beq\label{chi11}
\int_{\R^d} \left(F(x,\dfrac{x}{\ep}, f_1(x))-F(x,\dfrac{x}{\ep}, f_2(x))\right)(f_1(x)-f_2(x)) dx \geq 0.\eeq
A trivial approximation argument then yields  that \eqref{chi11} implies that, for each $\ep>0$, and all balls $B\subset \R^d$ and $f_1, f_2 :\R^d \to \R$ smooth and $\T^d$-periodic,
\[\int_{\R^d} \left(F(x,\dfrac{x}{\ep}, f_1(\dfrac{x}{\ep})\chi_B(x))-F(x,\dfrac{x}{\ep}, f_2(\dfrac{x}{\ep})\chi_B(x))\right)(f_1(\dfrac{x}{\ep})-f_2(\dfrac{x}{\ep}))\chi_B(x) dx \geq 0,\]
where $\chi_B$ is the characteristic function of $B$.
\smallskip

It follows, after letting $\ep\to0$, that, for all balls $B$ in $\R^d$ and  $f_1, f_2 :\R^d \to \R$ smooth and $\T^d$-periodic, that $F$ in \eqref{chi3} must satisfy
\[\int_B\int_{\T^d}(F(x,y,f_1(y))-(x,y,f_2(y))(f_1(y)-f_2(y))dy dx\geq 0,\]
and, hence, for all $x\in \R^d$ and $f_1, f_2 :\R^d \to \R)$ smooth and $\T^d$-periodic,
\beq\label{2.15}
\int_{\T^d}(F(x,y,f_1(y))-(x,y,f_2(y))(f_1(y)-f_2(y))dy dx\geq 0.
\eeq

\section{The mfg-cell problem and derivation of the averaged system}

The aim here is to use the ansatz \eqref{ansatz} to derive formally  the mfg-cell problem \eqref{cell}. For its rigorous analysis, that is, existence of ergodic constant and correctors, under the assumptions we introduce in the next two sections, we refer to \cite{LL3, LCdF}. 
\smallskip

The formal argument relies on  combining facts  from  the theory of the homogenization of viscosity solutions and the two scale convergence--see Lions, Papanicolaou and Varadhan \cite{LPV}), Allaire \cite{A}, Nguetseng \cite{N} and Goudon and Poupaud \cite{GP}.
\smallskip

Inserting the assumed expansion for $u^\ep$ in the first equation of \eqref{mfg1}, expanding in $\ep$, writing $y$ for $x/\ep$ and recording only the $\ep^0$-order term in the formal expansion in powers of $\ep$ leads to 
\[\partial_t\ol u -\Delta_y v + H(D_x\ol u +D_y v, y) -F(x,y, \ol m \mu)=0\]
and, hence, to 
\[\partial_t\ol u +\ol H(D\ol u,\ol m, x)=0,\]
where, for each $p\in \R^d$ and $\ol m\in R$, $v=v(y;p, \ol m)$ and $\mu=\mu(y; p,\ol m)$,
\beq\label{eq1}
-\Delta v + H(p+Dv,y)- F(x, y,\ol m \mu)=\ol H(p,\ol m, x).
\eeq
Note that although the dependence in $p$ and $m$ is separated in the cell problem, the effective Hamiltonian is not necessarily separated. 
\smallskip

Inserting the expansion for $m^\ep$ in the second equation of \eqref{mfg1}, and again ``expanding'' in powers of $\ep$ yields and using the previous notation we find as coefficients of $\ep^{-1}$ and  $\ep^0$ respectively 
the equations
\beq\label{takis200}
\Delta_y \mu +\text{div}_y\left[D_pH(D_x \ol u +D_y v, y) \mu\right]=0,
\eeq
and
\beq\label{takis201}
\begin{split}
&\ol m \Big(\Delta_y \nu +\text{div}_y\left[D_pH(D_x \ol u +D_y v, y)\nu\right]\Big)\\[2mm]
&= \mu \partial_t\ol m + \text{div}_x\big[D_pH(D_x \ol u +D_y v, y) \mu \big] +2 \left(D_x \ol m, D_y \mu\right). 
\end{split} 
\eeq
Rewriting \eqref{takis200} with $p\in \R^d$ in place of $D_x\ol u(x,t)$ and $\ol m \in \R$ in place of $\ol m(x,t)$
we find
\[\Delta_y \mu +\text{div}_y\left[D_pH(D_x \ol u +D_y v, y) \mu\right] =0,\]
which is the second equation of \eqref{cell}.
\smallskip

The existence of $\nu$ in \eqref{takis201} follows  from  Fredholm's alternative provided that the integral of the right hand side of \eqref{takis201}  over $\T^d$ vanishes. This requirement together with the normalization 
$\int_{\T^d} \mu(y)dy=1,$
give the transport equation of the homogenized system, namely
\[\partial_t \ol m + \text{div}_x\big[ \int_{\T^d} \left [D_pH\left(D\ol u + D_yv \right)\mu\right] dy  \ \ol m \big]=0.\]

We continue with the nonlocal setting, which, in view of the assumption that the forcing term $F_\ep$ is independent of $x/\ep$, decouples. Indeed, the classical periodic homogenization theory, implies that, in view of  \eqref{periodic} and    \eqref{convex}, for each $p\in \R^d$, there exists a unique $\ol H(p)$ such that 
\[-\Delta v + H(Dv+p, y)=\ol H(p),\]
has a unique up to constants periodic solution. 
\smallskip

It then follows that, for  every $p\in \R^d$, $x\in \R^d$ and $m \in \mathcal{P}$,   there exists a unique constant 
$\ol H(p,m, x)=\ol H(p) - F(x,m)$ and a unique pair $(v, \mu) \in C^2(\R^d) \times L^1(\T^d)$, which are $\T^d$-periodic and 
\begin{equation}\label{cell}
\begin{cases}
-\Delta v + H(Dv+p, y) -F(x, m)=  \ol H(p,m,x) \ \text{in} \ \T^d,\\[1.5mm]
\Delta \mu + \text{div} \left[ DH(Dv+p,y)\mu\right]=0  \ \text{in} \ \T^d,\\[1.5mm]
\int_{\T^d} v(y) dy=0,  \  \mu > 0  \  \text{in} \ \T^d \ \text{and} \  \int_{\T^d} \mu(y) dy=1.
\end{cases}
\eeq
 
The local cell-problem
\begin{equation}\label{cell.1}
\begin{cases}
-\Delta v + H(Dv+p, y) -F(y, m\mu) = \ol H(p,m) \ \text{in} \ \T^d,\\[1.5mm]
\Delta \mu + \text{div} \left[ DH(Dv+p,y)\mu\right]=0  \ \text{in} \ \T^d,\\[1.5mm]
\int_{\T^d} v(y) dy=0,  \  \mu > 0  \  \text{in} \ \T^d \ \text{and} \  \int_{\T^d} \mu(y) dy=1,
\end{cases}
\end{equation}
is a reparametrization of systems that have already solved in the literature, see, for example, \cite{LL3} and \cite{LCdF}, where we refer to for the details.

\section{Homogenization for mfg-systems with nonlocal coupling}

In addition to the conditions on $H$ in section~2, here we  assume that 
\begin{equation}\label{m1}
\text{there exist constants} \ C>c>0 \ \text{such that} \  c\leq \int_{\R^d}  m_T(y)dy \leq C,
\end{equation}
and 
\beq\label{u0}
u_0  \ \text{is  Lipschitz continuous.} 
\eeq

It follows from  \eqref{m1} and \eqref{u0} and the assumptions on $H$ and $F$ that, for some constants $C, c>0$,  
\beq\label{takis400}
 \underset{t\in [0,T]}\sup \|Du^\ep(\cdot,t)\| \leq C \ \  \text{and} \ \   c \leq \underset{t\in [0,T]}\inf \int_{\R^d} m^\ep(x,t) dx \leq \underset{t\in [0,T]}\sup \int_{\R^d} m^\ep(x,t) dx \leq C.
 \eeq
Since the bounds on $\int_{\R^d} m^\ep(x,t) dx $ are classical,  here we only discuss the Lipschitz estimate on $u^\ep$.
\smallskip

In view of \eqref{2.2}, \eqref{2.3} and \eqref{LIP}, such bounds  follow from a Bernstein-type argument, see, for example,  Lions and  Souganidis \cite{LS10}, provided it is shown that there exists $C>0$ such that, for all sufficiently small $\ep>0$,
\beq\label{chi20}
\|\partial_t u^\ep\|_\infty \leq C.
\eeq

Since it is immediate from that first equation of \eqref{mfg1} that, for all $t\in [0,T]$, 
\[\|\partial_t u^\ep(\cdot, t)\|_\infty \leq \|\partial_t u^\ep(\cdot, T)\|_\infty +\int_0^T \|\partial_s F_\ep[m^\ep] \|_\infty ds \leq \|\partial_t u^\ep(\cdot, T)\|_\infty +\int_0^T \|<F_\ep'[m^\ep], \partial_s m^\ep>\|_\infty ds,\]
it is enough to obtain a bound for the quantity in the last integral above. 
\small

This is where the regularity assumptions \eqref{2.2} and \eqref{2.3}. Indeed differentiating the second equation in \eqref{mfg1} in time and using \eqref{2.2} and \eqref{2.3} we find that $\|<F_\ep'[m^\ep], \partial_s m^\ep>\|_\infty$ is controlled by $\|D_pH(Du^\ep, \cdot)\|_\infty$, which,  in view of \eqref{bound} is bounded. 
\smallskip

We state next the homogenization result, which also implies that the limiting system is of mfg-type. 
\begin{theorem}\label{main1}
Assume that \eqref{mfg1} is nonlocal, that is, \eqref{2.1}, as well as \eqref{periodic}, \   \eqref{convex}, \  \eqref{bound},  \ \eqref{FLip}, \ \eqref{cm0},  \ \eqref{2.2}, \eqref{2.3},  \ \eqref{LIP}, \ \eqref{m1}.   
 and   \eqref{u0}.
Then there exists a Hamiltonian $\ol H=H(p,m):\R^d\times \mathcal{P} \to \R$, which satisfies \eqref{convex},  \eqref{bound}, \eqref{cm0} and \eqref{m1}, such that, as  $\ep\to 0$, 
$u^{\ep} \to \overline u$ locally uniformly in $\R^d$  and $m^{\ep}  \rightharpoonup \overline m$ in $L^\infty_t(L^1_x(\R^d))$, where  $(\overline u, \overline m)$ is the unique solution   \eqref{mfg2} with $\ol b$ satisfying \eqref{mfg10}.
\end{theorem}
Recall that  $(\ol u,\ol m)$ solves \eqref{mfg2} if $\ol u$ is a viscosity solution of the Hamilton-Jacobi equation and 
$\ol m$ a distributional solution of the transport equation in  in \eqref{mfg2}. 
\smallskip

We continue with the proof.

\begin{proof}[The proof of Theorem~\ref{main1}] The arguments are  
based on a combination of the viscosity-type homogenization theory for the Hamilton-Jacobi-Bellman equation in \eqref{mfg1} (see \cite{LPV}) and the double scale limit-method for the transport equation in \eqref{mfg1} (see  \cite{A},  \cite{N} and  \cite{GP}).  An additional argument is needed in the end to connect the two steps by showing the particular form of  the limiting transport equation.
\smallskip

 Throughout the proof all the limits $\ep\to 0$  will be taken along subsequences, a fact which will not be repeated. In addition, to keep the formulae shorter we often omit the explicit dependence of $u^\ep, Du^\ep$ and $m^\ep$ on $(x,t).$ 
 \smallskip
 
 We begin with the two-scale convergence argument.  In view of \eqref{takis400} and the assumptions on $H$, it follows that 
 the families  $(m^\ep)_{\ep>0}$ and $(D_pH(Du^\ep, \dfrac{\cdot}{\ep})m^\ep)_{\ep>0}$  is are equibounded n $L^1(\R^d \times [0,T])$.
\smallskip

Moreover, a simple calculation also yields that $(m^\ep)_{\ep>0}$ is weakly (in time) equicontinuous,  that is, for all
$\phi \in C_c(\R^d)$, the map $t\to \int_{\R^d} m^\ep(x,t) \phi(x) dx$ is continuous uniformly on $\ep$. 
 
\smallskip

The two-scale convergence theory yields that, along subsequences,  $(m^\ep)_{\ep>0}$ 
and $(D_pH(Du^\ep, \dfrac{\cdot}{\ep})m^\ep)_{\ep>0}$ ``double scale'' converge 
respectively to 
 $M: \R^d\times \R^d\times [0,\infty) \to \R$ and $\widetilde M=(\widetilde M_1, \ldots,\widetilde M_d): \R^d\times \R^d\times [0,\infty) \to \R^d$.  
\smallskip

It follows that, as $\ep\to 0$,  for every interval $I\subset (0,\infty)$,  
every smooth and $\T^d$-periodic  with respect to its second argument $\phi:\R^d\times \R^d\times [0,T]\to \R$ and for $i=1,\dots,d$,
\beq\label{takis2}
\int_0^T\int_{\R^d} m^\ep (x,t) \phi(x, \dfrac{x}{\ep}, t) dx dt \to 
\int_0^T\int_{\R^d}\int_{\T^d} M(x,t) \phi(x, y, t) dx dy dt, 
\eeq
\beq\label{takis2.1}
\int_0^T\int_{\R^d}H_{p_i}(Du^\ep, \dfrac{x}{\ep})m^\ep  \phi(x, \dfrac{x}{\ep}, t) dx dt \to 
\int_0^T\int_{\R^d}\int_{\T^d} \widetilde M_i(x,t) \phi(x, y, t) dx dy dt,
\eeq
and, moreover,
\beq\label{takis3}
m^\ep \underset{\ep\to 0 }\rightharpoonup \overline m=\ol m(x,t)=\int_{\T^d}  M(x,y,t) dy  \ \ \text{in} \ \  L^1(\R^d\times (0,T)).  
\eeq
\smallskip


In view of \eqref{u0},  \eqref{cm} and  \eqref{takis3} standard arguments from the theory of periodic homogenization of  Hamilton-Jacobi equations (see, for example, Lions, Papanicolaou and Varadhan~\cite{LPV}), yield that, along subsequences, the family $(u^\ep)_{\ep>0}$ converges locally uniformly to a solution of 
\beq\label{takis1.1}
\partial_t \overline u + \overline H(D\overline u) - F(x,\ol m)=0 \ \text{in} \ \R^d\times (0,T] \ \text{and} \ \overline u(\cdot,0)=u_0,
\eeq
where, for each $(p,m)\in \R^d \times  L^1(\R^d\times (0,T))$,   $\overline H=\overline H(p,m)$ is the effective Hamiltonian, that is,  unique constant such that the cell problem
\beq\label{cell}
-\Delta w + H(Dw+p,y)=\overline H(p) \ \text{in} \ \T^d
\eeq
has a periodic smooth solution, often referred to  as the corrector, which is unique if it is normalized by 
\beq\label{cell1}
 \int_{\T^d}w(y)dy =0. 
\eeq

Multiplying the transport equation in \eqref{mfg1} by $\phi(x) +\ep v(x, \frac{x}{\ep})$, where $\phi \in C^2_c(\R^d)$ and $v\in C^2_{c,p}(\R^d \times \T^d),$ and integrating over $\R^d$ we find 
\beq\label{takis3.1}
\begin{split} \dfrac{d}{dt} \int_{R^d} m^\ep \left[ \phi(x) + \ep  v(x, \frac{x}{\ep})\right] dx -\int_{R^d} \left( D_pH(Du^\ep,\frac{x}{\ep})m^\ep, D\phi(x)  +\left[\ep D_x +D_y]v(x, \frac{x}{\ep})\right]\right) dx \\
 =-\int_{R^d} m^\ep \left(\ep \Delta \phi (x) +[ \ep^2 \Delta_x + 2\Delta_{xy} +\Delta_y]v(x,\frac{x}{\ep})\right) dx.
\end{split}\eeq

Letting  $\phi\equiv 0$ in   \eqref{takis3.1} and using \eqref{takis2.1} yields,  after letting $\ep\to0$ and integrating over $[0,T]$, that, for all $v\in C^2_{c,p}(\R^d \times \T^d)$, 
\beq\label{takis4}
\int_0^T\int_{\R^d}\int_{\T^d} [(D_yv , \widetilde M) - \Delta_y v M] dx dy dt=0.
\eeq
Similarly, if $v\equiv 0$ in \eqref{takis3.1}, letting $\ep \to 0$ and using \eqref{takis3}, we get 
\beq\label{takis5}
\partial_t m  + \text{div}\int_{\T^d} \widetilde M(x,y,t)dy=0.
\eeq

To conclude we need to show that 
\beq\label{takis130.1}
\int_{\T^d} \widetilde M(x,y,t) dy = D_p \ol H(D\ol u) \int_{\T^d} M(x,y,t) dy=D_p \ol H(D\ol u)\ol m.
\eeq

We combine next \eqref{takis1.1} and \eqref{takis5} in the usual way, that is, we multiply \eqref{takis1.1} by $M$ and \eqref{takis5} by $u$ and integrate in $x$ and $y$. It follows that 
\beq\label{takis6}
\dfrac{d}{dt}\int_{\R^d}\int_{\T^d} \overline u(x,t) M(x,y,t)dy dx + \int_{\R^d}\int_{\T^d} \overline [\overline H(D\overline u)M -F(x,\overline m) M -(D\overline u , \widetilde M)]dydx=0.
\eeq

A similar caIculation at the $\ep$-level, which also classical in the mfg-theory yields, for each $\ep >0$, 
\beq\label{takis7} 
\dfrac{d}{dt} \int_{\R^d} u^\ep m^\ep dx +  \int_{\R^d} \left[H(Du^\ep, \frac{x}{\ep}) - F(x, m^\ep) -\Big(Du^\ep, D_pH(Du^\ep, \frac{x}{\ep})\Big)\right] m^\ep]dx=0.
\eeq 

The local uniform and weak convergence of $u^\ep$ and $m^\ep$ respectively give that, as $\ep\to 0$,
\[\int_0^T \int_{\R^d}u^\ep m^\ep dx dt \to \int_0^T \int_{\R^d}\int_{\T^d} \overline u M dydxdt.\]
Combining this last fact with  \eqref{takis6} and \eqref{takis7} we find, again as $\ep\to 0$, 
\beq\label{takis8}
\begin{split}
\int_0^T\int_{\R^d} \left[\Big( H(Du^\ep, \frac{x}{\ep}) -F(x, m^\ep)\Big)m^\ep-\Big(Du^\ep, D_pH(Du^\ep, \frac{x}{\ep})m^\ep \Big) \right] dx dt \to \\[2mm]
\int_0^T\int_{\R^d}\int_{\T^d} \left[\left(\overline H(D\overline u) -F(x,\overline m) \right)M -D\overline u \cdot \widetilde M \right] dydx.
\end{split}
\eeq
It also follows from the  definition of the  two-scale limit that, as $\ep\to0$,
\beq\label{takis9}
\int_0^T\int_{\R^d} H(D\overline u+Dw(\frac{x}{\ep}), \frac{x}{\ep})m^\ep dx dt \to \int_0^T\int_{\R^d} \int_{\T^d}
H(D\overline u +Dw)M dy dx dt,
\eeq
and 

\beq\label{takis10}
\int_0^T\int_{\R^d} \Big(D_pH(D\overline u+Dw(\frac{x}{\ep}), \frac{x}{\ep})m^\ep, (D\overline u + Dw(\frac{x}{\ep}) )\Big) \to   \int_0^T\int_{\R^d} \int_{\T^d}
\left(\widetilde M,(D\overline u+Dw)\right) dy dx dt.
\eeq
\smallskip

Combining all the previous facts  we find that, as $\ep \to 0$,
\beq\label{takis11}
\begin{split}
\int_0^T\int_{\R^d}  \Big[H(D\overline u &+ Dw(\frac{x}{\ep}),\frac{x}{\ep})-H(Du^\ep,\frac{x}{\ep})\\[2mm]
&-\left( D_pH (Du^\ep,\frac{x}{\ep}), D\overline u + Dw(\frac{x}{\ep}) -Du^\ep \right ) \Big] m^\ep dx dt \to 0.
\end{split}
\eeq
The uniform convexity of $H$ and \eqref{takis11} then  yield that, as $\ep\to 0$,
\beq\label{takis12}
\int_0^T\int_{\R^d}  \big|D\overline u +Dw(\frac{x}{\ep})-Du^\ep\big |^2 m^\ep dx dt \to 0,
\eeq
and, hence, as $\ep \to 0$, 
\beq\label{takis13}
\int_0^T\int_{\R^d}  \Big|D\overline u +Dw(\frac{x}{\ep})-Du^\ep\Big| m^\ep dx dt \to 0.
\eeq

\smallskip
The last claim follows from the observation that, if, for some $C>0$,   $|G^\ep| \leq C$ in $\R^d\times [0,T]$ and, as $\ep\to 0$, 
\begin{equation*}
\int_0^T\int_{\R^d} |G^\ep(x,t)|^2 m^\ep(x,t) dxdt \to 0,
\end{equation*}
then 
\begin{equation*}
\begin{split}
\int_0^T\int_{\R^d} |G^\ep(x,t)| m^\ep(x,t) dxdt& \leq \int_0^T\Big[\int_{\R^d} |G^\ep(x,t)|^2 m^\ep(x,t) dx\Big]^{1/2}\Big[\int_{\R^d}m^\ep(x,t)dx\Big]^{1/2} dt\\[2mm]
& \leq \int_0^T\Big[\int_{\R^d} |G^\ep(x,t)|^2 m^\ep(x,t) dx\Big]^{1/2} dt \underset{\ep\to 0} \to 0.
\end{split}
\end{equation*}
The Lipschitz continuity of$DH$ and \eqref{takis13} then give  that, as $\ep\to 0 $, for all  $\phi \in C^2_{c, p}(\R^d\times \T^d)$ and $i=1,\ldots,d$,
\beq\label{takis14}
\int_0^T\int_{\R^d} \left[H_{p_i}(Du^\ep,\frac{x}{\ep})- H_{p_i}(D \overline u + Dw(\frac{x}{\ep}), \frac{x}{\ep})\right]m^\ep  \phi(x,\frac{x}{\ep}) dx dt \to 0. 
\eeq
Using  $H_{p_i}(D \overline u + Dw(y), y)\phi (x,y,t)$ as test function in the double-scale limit of $m^\ep$  we find that, as $\ep \to0$, for all $\phi \in C_{c, p}(\R^d\times \T^d)$ and $i=1,\ldots,d$, 
%
%
%
\[\int_0^T\int_{\R^d}[H_{p_i}(Du^\ep,\frac{x}{\ep})m^\ep \phi(x,\frac{x}{\ep},t)dxdt \to \int_0^T\int_{\R^d}\int_{\T^d} H_{p_i}(D\overline u +D_y w, y) M(x,y,t)\phi(x,y,t) dxdydt,\]
and, hence, in view of \eqref{takis2.1}, 

\beq\label{takis15}
\widetilde M=D_p H(D\overline u +Dw, y) M. 
\eeq

It follows then follows from   \eqref{takis5} and \eqref{takis14} that 
 
\beq\label{takis16}
\partial_t \overline m + \text{div}_x\left(\int_{\T^d} D_pH(D\overline u + Dw, y)M(x,y,t) dy\right )= 0 \  \text{in} \ \R^d\times [0,T). 
\eeq

It remains to show that

\beq\label{takis410}
\int_{\T^d} D_pH(D\overline u + Dw, y)M(x,y,t)dy= \overline H_{p_i}(D\overline u) \overline m.\eeq

Differentiating \eqref{cell} with respect to $p$ we find, 
 for each $i=1,\ldots, d$,
 \beq\label{takis16.1}
 -\Delta w_{p_i} + D_pH (p+D_y w, y) \cdot Dw_{p_i} + H_{p_i}=\overline H_{p_i} \  \text{in} \ \R^d.
 \eeq

The argument above can be justified by taking difference quotients in \eqref{cell} and passing in the limiting using the 
continuity properties of the ergodic constant.  We leave the details to the reader.
\smallskip

Let $p=D\overline u(x,t)$ in \eqref{takis16.1}, multiply by $\phi M$ with $\phi \in C_c(\R^d \times [0,T])$, integrate over $\R^d \times \T^d \times [0,T]$ and use \eqref{takis4} to get 
\[ \int_0^T\int_{\R^d}\int_{\T^d} \phi(x) H_{p_i}(D_x \overline u +D_y w, y) M dydxdt=\int_0^T\int_{\R^d}\int_{\T^d} \overline H_{p_i}(D\overline u)  Mdydxdt,\] 
hence, \eqref{takis410} and \eqref{takis130.1} hold, and, thus 

\[\partial _t \overline m+ \text{div} (D_p \overline H(D\overline u)\overline m)=0  \  \text{in} \ \R^d\times [0,T). \]

\end{proof}

We turn next to the well-posedness.

\begin{proof}[The proof of Theorem~\ref{wp1}]
The fact that $F$ is independent of $y$  and the uniqueness of the ergodic constant $\ol H$ yield that, for each $p\in \R^d$ and $m\in L^1(\R^d)$,  
\beq\label{takis300}
\ol H(p, m)= \ol h(p) - F(m),
\eeq
where $\ol h(p)$ is the unique constant for which the cell problem
\[-\Delta_y w + H(p+D_y w, y)=\ol h(p) \ \text{in} \ \R^d\]
has a $\T^d$-periodic solution.
\smallskip

It is well known, see, for example, \cite{LPV}, that the convexity of $H$ implies the convexity of $\ol h$. Since $F$ is the same as in \eqref{mfg1},  $\ol H$ has the properties needed for the limit mfg-system to have a unique solution.

\end{proof}

\section{Homogenization for mfg-systems with local coupling}

When dealing with \eqref{mfg1} with local coupling we loose the regularizing property which allowed  in the previous section to pass the weak limit of the $m^\ep$'s in the nonlinearity, which, in turn, essentially decoupled the cell-problem system.  
\smallskip 

Instead of introducing conditions that lead in some cases to independent of $\ep$ apriori bounds on $(u^\ep, m^\ep)$, here we concentrate on proving a homogenization result for \eqref{mfg1} assuming that the limit system has a classical solution, which is true when the  horizon $T$ using is small. 
\smallskip

We work with well-prepared initial and terminal conditions, that is, we assume that 
\beq\label{wp}
u^\ep_0(x)=\ol u_0(x) +\ep v(\frac{x}{\ep})  \ \ \text{and} \ \  m^\ep_T(x)=\ol m_T(x)(\mu(\frac{x}{\ep}) + \ep \nu(\frac{x}{\ep}),
\eeq
where $\ol u_0$ and $\ol m_T$ satisfy respectively \eqref{u0} and \eqref{m1} and $v, \mu$ and $\nu$ are constructed from $\ol u_0$ and $\ol m_T$  as in the course fo the proof of Theorem~\ref{local2}.

\begin{theorem}\label{local2} Assume that \eqref{mfg1} is local, that is, \eqref{chi4} and \eqref{chi5}, and, in addition,  \eqref{periodic}, \eqref{convex}, \eqref{bound}, \eqref{local}, either \eqref{2.8} or \eqref{2.9}   and fix 
$\ol u_0 \in C^2_b(\R^d) \cap H^1(\R^d)$ and $\ol m_T \in C^2_b(\R^d) \cap L^1(\R^d)$ satisfying respectively \eqref{u0} and \eqref{m1}. Moreover, assume  that the system \eqref{mfg2} with initial and terminal condition $\ol u_0$ and $\ol m_T$ respectively and  $\ol H$ as in \eqref{cell1} and $\ol b$ given by \eqref{mfg2.1} has a classical solution $(\ol u,\ol m)$. Finally, let $(u^\ep, m^\ep)$ be the solution of \eqref{mfg1} with initial and terminal condition $\ol u^\ep_0$ and $\ol m^\ep_T$ as in \eqref{wp}.   Then, as $\ep\to 0$, $u^\ep \to \ol u$  in $H^1(\R^d\times (0,T))$ and   $m^\ep \to \ol m$ in $ L^1(\R^d)$.     
\end{theorem} 
\begin{proof}
The proof is based on constructing, using the regularity of  $(\ol u, \ol m)$ and the ansatz \eqref{ansatz}, an approximate solution of \eqref{mfg1}, which is very close to $(u^\ep, m^\ep)$.  
\smallskip

The assumption that on the initial and terminal conditions of \eqref{mfg1} allows for a significant simplification of the proof, which, nevertheless is tedious. 
\smallskip

We assume that the solution  $(v(\cdot; p, m), \mu(\cdot; p,m))$ of \eqref{cell} depends smoothly on $(p,m)$, introduce the functions $\tilde v$ and $\tilde \mu$ given by 
\[\tilde v(y,x,t)= v(y; D\ol u(x,t), \ol m(x,t)) \ \ \text{and} \ \   \tilde \mu (x,y,t)= \mu(y; D\ol u(x,t), \ol m(x,t)),\] 
and claim that there exists a smooth $\tilde \nu: \T^d \times \R^d\times [0,T] \to \R$ such that 
\beq\label{takis50}
\widehat u^\ep(x,t)=\ol u(x,t) + \ep \tilde v(\frac{x}{\ep}, x,t) \ \ \text{and} \ \  \widehat m^\ep(x,t)= \ol m(x,t)\left(\tilde\mu (\frac{x}{\ep}, x,t) + \tilde \nu(\frac{x}{\ep}, x,t)\right)
\eeq

is a solution of  \eqref{mfg1} up to an error, which is small in a sense that will become clear in the course.
\smallskip

Notice that 
\[u_0(x)= \widehat u^\ep (x,0)  \ \text{and}  \ m_T(x)=   \widehat m^\ep(x,T).\]

We begin with the first equation in \eqref{mfg1}. In what follows, we use that 
\[ \partial_t \ol u+ \ol H(D\ol u, \ol m)=0 \ \ \text{and} \ \  -\Delta_y v + H(D\ol u + D_y v, y0- F(y,  \ol m \mu)= \ol H(D\ol u, \ol m)\]
and omit, whenever it does not create confusion,  the dependence of all functions on their arguments. 
\smallskip

Inserting the formula for $\wh u^\ep$ in the Bellman equation of \eqref{mfg1} we find 
\smallskip

\begin{equation*}
\begin{split}
&\partial_t \wh u^\ep-\ep \Delta \wh u^\ep +H(D \wh u^\ep, \frac{x}{\ep}) - F(\frac{x}{\ep},\wh m^\ep)=
\partial_t \ol u + \ep \partial_t \tilde v - \ep \Delta \ol u 
-  \left( \Delta_y\tilde v + 2\ep \Delta_{x,y} \tilde v
+\ep^2 \Delta_x \tilde v \right)\\[2mm]
& + H\left(D\ol u + D_y \tilde v  + \ep D_x \tilde v,\frac{x}{\ep}\right) -F\left (\frac{x}{\ep}, \ol m ( \tilde\mu+ \ep \tilde \nu)\right) = \partial_t \ol u + \ol H(D\ol u, \ol m) \\[2mm]
&- \Delta_y v(\frac{x}{\ep}; D\ol u, \ol m) + H\left( D\ol u +D_y v(\frac{x}{\ep}; D\ol u, \ol m),  \frac{x}{\ep}\right) -F\left(\frac{x}{\ep}, \ol m \mu(\frac{x}{\ep};  D\ol u, \ol m(x,t)\right) - \ol H(D\ol u, \ol m)= A^\ep(x,t), 
\end{split}
\end{equation*}
where 
\beq\label{takis51}
\begin{split}
&A^\ep(x,t) =  \ep \partial_t \tilde v -\ep \Delta \ol u - 2 \ep \Delta_{x,y} \tilde v - \ep^2 \Delta_x \tilde v + H\left(D\ol u + D_y v + \ep D_x \tilde v,\frac{x}{\ep}\right) \\[2mm] 
&  -F\left(\frac{x}{\ep}, \ol m( \tilde\mu+ \ep \tilde \nu)(\frac{x}{\ep}, x,t)\right) 
 - H\left( D\ol u +D_y v,  \frac{x}{\ep}\right) + F\left(\frac{x}{\ep}, \ol m \mu)\right);
\end{split}
\eeq
note that above  we used  the fact  that $D_y \tilde v= D_y v$. 
\smallskip

The assumptions on $H$ and $F$ and the regularity of $\ol u$ and $\ol m$ yield  $A_1^\ep \in C_b( \R^d\times \R; \R^d)$ and 
$A_2^\ep \in C_b( \R^d\times \R;  \R)$ such that 
\[ 
\begin{split}
&H\left(D\ol u +(D_y  + \ep D_x) \tilde v,\frac{x}{\ep}\right) - F\left(\frac{x}{\ep}, \ol ( \tilde\mu+ \ep \tilde \nu)(\frac{x}{\ep}, x,t)\right)=\\[2mm]
& H\left( D\ol u +D_y v(\frac{x}{\ep}; D\ol u, \ol m),  \frac{x}{\ep} \right) -F\left(\frac{x}{\ep}, \ol m \mu(\frac{x}{\ep};  D\ol u, \ol m) \right) + 
\ep \left( A^\ep_1, D_x \tilde v\right) + \ep A^\ep_2 \tilde \nu,
\end{split}
\]
and, hence, for some $C>0$,
\[|A^\ep|\leq C \ \text{on} \ \R^d\times [0,T],\]

that is, 
\beq\label{takis52}
\|\partial_t \wh u^\ep-\ep \Delta \wh u^\ep +H(D \wh u^\ep, \frac{x}{\ep}) -F(\frac{x}{\ep}, \wh m^\ep)\|_ {L^\infty(\R^d\times [0,T])} \leq C \ep.
\eeq

Justifying the above is a long  but nevertheless routine calculus exercise, which requires knowing 
that $\tilde \nu$ as well as $D_p v, D_m v, D_p \mu$ and $D_m\mu$ are uniformly bounded. The existence of such 
$\tilde \nu$ is discussed below. The other bounds follow from differentiating \eqref{cell} with respect to $p$ and $m$ and studying the resulting problems. We leave these details up to the reader.
\smallskip

The argument for the second equation of \eqref{mfg1} is slightly more complicated since it involves the two scale convergence. To simplify the notation, we find it necessary to have some preliminary discussion and to introduce some useful notation.  As before, when it does not create confusion, we omit the explicit dependence on the variables.
\smallskip

We begin with the observation that the independent of $\ep$ bounds on $D_x \tilde v$ and $D_x\tilde \mu$ imply that 
%
%
there exist $b^\ep_1, b^\ep_2 \in C^1_b(\R^d\times [0,T]; \R^{d})$ with $C^1$-bounds independent of $\ep$  such that 
\beq\label{takis60}
D_pH(D\wh u^\ep, \frac{x}{\ep})=D_pH(D_x \ol u + D_y v + \ep D_x\tilde v, \frac{x}{\ep})
= D_pH (D\ol u + D_y v, \frac{x}{\ep}) + \ep b^\ep_1  +  \ep^2 b^\ep_2. 
\eeq

We return now to the second equation of \eqref{mfg1}. In what follows, to simplify the notation we write $\mu$ and $\nu$ in place of $\tilde \mu$ and $\tilde \nu$.  We find that 
\[
\begin{split}
&\partial_t \wh m^\ep + \ep \Delta \wh m^\ep + \text{div} \big[ D_p H(D\wh u^\ep) \wh m^\ep]=\\[1mm]
& \partial_t \ol m ( \mu +\ep  \nu) + \ol m \partial_t ( \mu +\ep  \nu) + \ep \Delta_x \ol m ( \mu +\ep  \nu) +  \ol m (\frac{1}{\ep} \Delta_y + 2 \Delta_{xy} + \ep \Delta_x)(\mu +\ep \nu) + 2 \left( D_x\ol m, D_y( \mu +\ep  n)\right) \\[1.5mm]
& + 2\ep  \left( D_x\ol m, D_x( \mu +\ep  n)\right)+  \frac{1}{\ep}\text{div}_y \big[ D_p H(D \ol u + D_y v,  \frac{x}{\ep})\ol m \mu\big] + \text{div}_x \big[ D_p H(D \ol u + D_y v,  \frac{x}{\ep}) \ol m \mu \big]  \\[2mm]
& + \text{div}_y \big[ D_p H(D \ol u + D_y v,  \frac{x}{\ep})\ol m \nu\big] + \ep \text{div}_x \big[ D_p H(D \ol u + D_y v,  \frac{x}{\ep})\ol m  \nu \big] + \text{div}_y [b^\ep_1 \ol m  \mu]  \\[2mm]
& + \ep \text{div}_x [b^\ep_1 \ol m  \mu] + \ep\left(\text{div}_y [b^\ep_1 \ol m \nu] +\ep \text{div}_x [b^\ep_1 \ol m  \nu]\right)+ \ep \left(\text{div}_y [b^\ep_2 \ol m \mu] +\ep \text{div}_x [b^\ep_2 \ol m  \mu]\right) \\[2mm]
& + \ep^2 \left(\text{div}_y [b^\ep_2 \ol m \nu] +\ep \text{div}_x [b^\ep_2 \ol m  \nu]\right).
\end{split}
\]

Reorganizing  the identities above yields
\[
\partial_t \wh m^\ep + \ep \Delta \wh m^\ep + \text{div} \big[ D_p H(D\wh u^\ep) \wh m^\ep] 
= \frac{1}{\ep}B^\ep_1 + B^\ep_2 + \ep B^\ep,
\]
where
\[
B^\ep_1 =  \ol m \left(\Delta_y \mu + \text{div}_y \big[ D_p H(D \ol u + D_y v,  \frac{x}{\ep}) \mu\big] \right), \]
\[
\begin{split}
B^\ep_2&= \ol m \big(\Delta_y \nu + \text{div}_y \big[ D_p H(D \ol u + D_y v,  \frac{x}{\ep}) \nu\big] + \partial_t \mu + 2 \Delta_{xy} \mu + 2 \left( D_x \ol m, D_y \mu \right) \big) \\[1.5mm]
&+ \partial_t \ol m \mu +  \text{div}_x \big[ D_p H(D \ol u + D_y v,  \frac{x}{\ep}) \ol m \mu \big] + \text{div}_y [b^\ep_1 \ol m  \mu],
\end{split} 
\]
and 
\[
\begin{split}
B^\ep_3=& \partial_t \ol m \nu + \ol m \partial_t \nu + \Delta \ol m (\mu +\ep \nu) +\ol m \Delta_x(\mu + \ep \nu) +2\left( D_x\ol m, D_y \nu\right)  + 2 \left( D_x \ol m, D_x(\mu + \ep \nu)\right) \\[1.5mm]
& + \text{div}_x \big[ D_p H(D \ol u + D_y v,  \frac{x}{\ep})\ol m  \nu \big]  +  \text{div}_x [b^\ep_1 \ol m  \mu] +  \text{div}_y [b^\ep_1 \ol m \nu]
+ \text{div}_x [b^\ep_1 \ol m  \nu]   \\[2mm]
& +  \left(\text{div}_y [b^\ep_2 \ol m \mu] + \text{div}_x [b^\ep_2 \ol m  \mu]\right)
 + \ep \left(\text{div}_y [b^\ep_2 \ol m \nu] +\ep \text{div}_x [b^\ep_2 \ol m  \nu]\right).
\end{split}
\]

The choice of $\mu$ implies that $B^\ep_1=0$. Moreover, using Fredholm's alternative it is also possible to find a $\T^d$-periodic $\nu$ such that 
$B^\ep_2=0.$ Indeed, the choice of $\ol m$ and the fact that $\int_{\T^d} \mu =1$ imply that 
\[\int_{\T^d} \Big(\partial_t \mu + 2 \Delta_{xy} \mu + 2 \left( D_x \ol m, D_y \mu \right) 
+ \frac{1}{\ol m} \big[\partial_t \ol m \mu +  \text{div}_x \big[ D_p H(D \ol u + D_y v,  \frac{x}{\ep}) \ol m \mu \big] + \text{div}_y [b^\ep_1 \ol m  \mu]\big]\Big)=0. \]

Finally, it is immediate from the assumed regularity of $\ol m$ that $B^\ep_3\L^1(\R^d\times [0,T])$ with norm bounded independently of $\ep$, and, hence, there exists $C>0$ such that 
\beq\label{takis70}
\|\partial_t \wh m^\ep + \ep \Delta \wh m^\ep + \text{div} \big[ D_p H(D\wh u^\ep) \wh m^\ep]\|_{L^1(\R^d\times [0,T])}\leq C\ep.
\eeq

\smallskip

%
%

Next we compare $(u^\ep,  m^\ep)$ to $(\wh u^\ep, \wh m^\ep)$ using the typical in the mfg-theory argument consisting of writing the equations for $u^\ep-\wh u^\ep$ and $m^\ep-\wh m^\ep$. It follows that  

\beq\label{takis56}
\begin{cases}
\partial_t (u^\ep-\wh u^\ep) - \ep \Delta (u^\ep-\wh u^\ep)\\[2mm]
\hskip.5in  + H(Du^\ep, \frac{x}{\ep}) -F(\frac{x}{\ep}, m^\ep)-H(D\wh u^\ep,  \frac{x}{\ep}) + F(\frac{x}{\ep}, \wh m^\ep)= \text{O}_{L^\infty}(\ep) \ \text{in} \ \R^d\times (0,T], \\[2mm]
\partial_t (m^\ep -\wh m^\ep) + \ep \Delta (m^\ep-\wh m^\ep) +\\[2mm]
\hskip.5in  \text{div}\Big[D_p H(Du^\ep, \frac{x}{\ep}) m^\e) - D_p H(D\wh u^\ep, \frac{x}{\ep}) \wh m^\ep\Big]= \text{O}_{L^1}(\ep) \ \text{in} \ \R^d\times (0,T], \\[2mm]
(u^\ep-\wh u^\ep)(\cdot, 0)=0 \ \text{and} \ (m^\ep -\wh m^\ep)(\cdot, T)=0.
\end{cases}
\eeq

Multiplying the first equation of \eqref{takis56} by $m^\ep -\wh m^\ep$ and the second by $u^\ep-\wh u^\ep$, integrating over $\R^d\times[0,T]$, adding the two integrals and using the initial and terminal conditions in \eqref{takis56} yields
\beq\label{takis57}
\begin{split}
&\int_0^T\int_{\R^d} \Big[(m^\ep -\wh m^\ep)(H(Du^\ep, \frac{x}{\ep})- F(\frac{x}{\ep}, m^\ep)-H(D\wh u^\ep,  \frac{x}{\ep}) +F(\frac{x}{\ep}, \wh m^\ep))\\[2mm]
-&\left( D(u^ep-\wh u^\ep), D_p H(Du^\ep, \frac{x}{\ep}) m^\e) -D_p H(D\wh u^\ep, \frac{x}{\ep}) \wh m^\ep \right)\Big] dx dt  = \text{O}(\ep).
\end{split}
\eeq

It follows from the uniform convexity of $H$, the monotonicity with respect to $m$ and the assumed strict positivity of $\ol m$, which yields a positive and independent of $\ep$ lower bound for $\wh m^\ep$,  that
$\lim_{\ep\to 0} \|D(u^\ep-\wh u^\ep)\|_{L^2}$ and,  since $u^\ep-\wh u^\ep(\cdot, 0)$, that
$\lim_{\ep\to 0} \|u^\ep-\wh u^\ep\|_{H^1}$, and, then,   that $m^\ep-\wh m^\ep \to 0$ in $L^1$.

\end{proof}

We investigate next  whether  the limit system \eqref{mfg2} is of mfg-type or not, that is, if   \eqref{mfg10} holds. The answer is, in general, negative at least when dealing with potential mfg as we explain next.
\smallskip

We consider the mfg
\beq\label{potential0}
\begin{cases}
\partial_t u^\ep + \ep \Delta u^\ep + H(Du^\ep, \frac{x}{\ep})-F(\frac{x}{\ep}, m^\ep)=0 \ \text{in} \ \R^d \times [0,T),\\[1.5mm]
\partial_t m^\ep - \ep \Delta m^\ep + \text{div}\left[DH(Du^\ep,\frac{x}{\ep})m^\ep\right] \ \text{in} \ \R^d \times (0,T],\\[1.5mm]
u^\ep(\cdot,T)=u^\ep_T \ \text{and} \ m^\ep(\cdot, 0)=m_0,
\end{cases}
\eeq
and assume that 
\beq\label{potential1}
\text{there exists a $\T^d$-periodic and convex $\mathcal F:\R^d \times \R$ such that} \ F(y,m)=\frac{\delta \mathcal F}{\delta m}(y,m;
\eeq
notice that to be consistent with the classical stochastic control formulation in \eqref{potential0} we reversed the time.
\smallskip

It was shown in \cite{LL3} and Lions~\cite{LCdF} that 
\beq\label{potential2}
\begin{split}
u^\ep(x,0)=\inf\big\{\int_0^T\int_{\R^d}& \left[H^\star(a, \frac{x}{\ep}) m + \mathcal F(\frac{x}{\ep},m) \right] dx dt + \int_{\R^d}u_T(x) m(x,T) dx: \\[2mm]
&  \partial_t m-\ep \Delta m + \text{div}[am] =0 \ 
m(\cdot,0)=m_0\big\},
\end{split}
\eeq
where $H^\star$ is the convex conjugate of $H$.
\smallskip

The variational formula for $u^\ep$ is amenable to the $\Gamma$-convergence techniques, which yield a homogenized limit $\ol u$ for the $u^\ep$, which also satisfies, for an effective Hamiltonian $\ol H$,  the  variational formula
\beq\label{potential3}
\begin{split}
\ol u(x,0)=\inf\big\{\int_0^T\int_{\R^d} {\ol H}^\star(a,m) dx dt  &+ \int_{\R^d} u_T(x) m(x,T) dx: \\[2mm]
& \partial_t m-\Delta m + \text{div}[am] =0 \ 
m(\cdot,0)=m_0\big\},
\end{split}
\eeq

\smallskip

The limit $\ol u$ should also satisfy the first of the two equations of \eqref{mfg2} and, of course, the question is if \eqref{mfg10} holds or not. This is equivalent to checking if the limiting variational formula corresponds to a potential-type mfg.
\smallskip

It is known (see \cite{LCdF}), however, that for a mfg to have a potential formulation, the Hamiltonian must be of separated form, that is we must have that 
\[\ol H(p,m)= \ol H (p) + \ol F(m).\]
It is a classical fact that the $\Gamma$-limit of variational problems with Lagrangians corresponding to separated Hamiltonians does not give, in general,  rise to separated effective  Hamiltonians; see, for example, \cite{DM}.   Hence, the homogenized system is not, in general, of mfg-type.

\smallskip

The last item we discuss here is the rigorous convergence of the solution $(u^\ep, m^\ep)$ of \eqref{mfg1} when the initial and terminal conditions satisfy \eqref{well}. This is, however, a classical fact in the homogenization theory, where now the ansatz \eqref{ansatz} can be shown to hold rigorously as above,  given that, as follows form a simple calculation,   the solution of  homogenized system \eqref{mfg2} with initial and terminal conditions as in  \eqref{well} is given by 
\[\ol u(x,t)= \left( p, x\right) - t\ol H(p,1) \ \text{and} \ \ol m\equiv 1.\].

\bibliographystyle{amsplain}
\bibliography{mfg_homogenization.bib}

\bigskip

\noindent ($^{1}$) Coll\`{e}ge de France,
11 Place Marcelin Berthelot, 75005 Paris, 
and  
CEREMADE, 
Universit\'e de Paris-Dauphine,
Place du Mar\'echal de Lattre de Tassigny,
75016 Paris, FRANCE\\ 
email: lions@ceremade.dauphine.fr
\\ \\
\noindent ($^{2}$) Department of Mathematics 
University of Chicago, 
5734 S. University Ave.,
Chicago, IL 60637, USA, \ \  
email: souganidis@math.uchicago.edu
\\ \\

%

\end{document}